\documentclass{amsart}
\usepackage{amssymb, amsmath, amsfonts, amsthm}
\usepackage{graphics, enumerate, mathrsfs, mathtools,tikz-cd,soul,csquotes,dsfont}
\usetikzlibrary{positioning,arrows,scopes}
\usepackage[color=red!40]{todonotes}
\usepackage{bm,dsfont}

\definecolor{cadmiumgreen}{rgb}{0.0, 0.42, 0.24}
\usepackage[
	colorlinks, citecolor=cadmiumgreen,
	pagebackref,
	pdfauthor={Harry Richman, Farbod Shokrieh, and Chenxi Wu}, 
	pdfstartview ={FitV},
	bookmarksdepth=3,
]{hyperref}
\hypersetup{
	pdftitle={On determinants of resistance matrices},
}

\usepackage[
	backrefs,
	nobysame,
	lite,
	non-sorted-cites,
]{amsrefs} 

\newtheorem{thm}{Theorem}
\newtheorem*{thm*}{Theorem}

\newtheorem{prop}[thm]{Proposition}

\newtheorem{cor}[thm]{Corollary}

\theoremstyle{definition}

\newtheorem{dfn}[thm]{Definition}
\newtheorem{eg}[thm]{Example}
\newtheorem{rmk}[thm]{Remark}

\newcommand{\RR}{\mathbb{R}}

\newcommand{\bone}{\mathbf{1}}
\newcommand{\boldmu}{\bm{\mu}}
\newcommand{\boldtau}{\bm{\tau}}

\newcommand{\boldu}{\mathbf{u}}

\DeclareMathOperator{\cof}{cof}
\newcommand{\trees}{\mathcal{T}}
\newcommand{\forests}{\mathcal{F}}

\newcommand{\tr}{\intercal}

\begin{document}
\title{On determinants of resistance matrices}
\author{Harry Richman}
\author{Farbod Shokrieh}
\author{Chenxi Wu}
\date{November 29, 2025.}
%\thanks{This work was partially supported by ....}

\begin{abstract}
We prove a new combinatorial identity for the determinant 
of the resistance matrix of a finite graph, 
which involves counts of spanning trees and forests.
This generalizes a result of Graham and Pollak on distances matrices of trees.
We make use of Bapat's expression of the resistance matrix determinant as a linear algebraic quantity.
\end{abstract}


\maketitle

%\setcounter{tocdepth}{1}
%\tableofcontents

\section{Introduction}

Suppose $G = (V,E)$ is a finite connected graph.
We allow $G$ to have parallel edges, but not loops.
The {\em resistance matrix} of $G$ is defined as the matrix in $\RR^{|V| \times |V|}$ 
having entries
\[
    R_{ij} = \text{effective resistance between $v_i$ and $v_j$ in $G$},
\]
when $G$ is considered as an electrical network with a unit resistor on each edge.
% Let $D$ denote the distance matrix of $G$.
Combinatorially, this quantity is equal to a ratio of spanning tree counts~\cite{kirchhoff},
\[
    R_{ij} = \frac{\kappa(G/ij)}{\kappa(G)} ,
    % = \frac{\kappa_2(v_i | v_j)}{\kappa(G)} ,
\]
where $\kappa(G)$ denotes the number of spanning trees
and $G/ij$ denotes the graph obtained by identifying (or ``gluing together'') vertices $v_i$ and $v_j$.

In this paper, we provide a combinatorial formula for the determinant of the resistance matrix.
Given a finite graph $G$, 
% let $\kappa(G)$ be as above, and %denote the number of spanning trees, and 
let $\kappa_2(G)$ denote the number of two-component spanning forests.

\begin{thm}
\label{thm:main}
Let $G = (V, E)$ be a graph with $n$ vertices 
and resistance matrix $R$.
Then
\begin{equation}\label{eq:main}
\det R = \frac{(-1)^{n-1} {2^{n}}}{\kappa(G)} \left( \frac{1}{3}  \frac{\kappa_2(G)}{\kappa(G)}  - \frac1{12} \sum_{e \in E}  
% R_{e^+ \!, e^-}^2 \right) 
\frac{\kappa (G/ e)^2}{\kappa(G)^2}   \right)
\end{equation}
%where $F(e) = 1 - r(e^+, e^-)$ is the Foster coefficient.
where 
$G/ e$ is the edge contraction of $G$ at $e$. 
\end{thm}

This theorem is proved by combining two existing results in the literature: one by Bapat~\cite{bapat-resmatrix} on the determinant of the resistance matrix,
and one by the current authors~\cite{RSW-two-forest} on identities involving counts of spanning forests and resistances.
For some examples of resistance matrices and their determinants, see \S\ref{sec:examples}.

% Alternative form of theorem:
We also give a second alternative expression for the determinant of the resistance matrix that is less explicitly combinatorial.
\begin{thm}\label{thm:main-alt}
Let $R$ be the resistance matrix of a graph $G = (V, E)$ with $n$ vertices.
Then for any vertex index $k \in \{1, \ldots, n\}$, %vertex $q \in V$,
\begin{equation}\label{eq:main-alt}
	\det R = \frac{(-1)^{n - 1} 2^{n - 2}}{\kappa(G)} \sum_{\substack{e \in E \\ e = (i, j)}} \left(R_{ik} - R_{jk}\right)^2 . 
	%(R_{e^+ \!, q} - R_{e^- \!, q})^2 .
\end{equation}
\end{thm}
Here we use ``$e = (i, j)$'' to mean that $e$ connects vertices $v_i$ and $v_j$, by abuse of notation.
This result states that, surprisingly, the computation of $\det R$ may be reduced to an expression using only entries of any {\em single column} of $R$, up to the term $\kappa(G)$.

For a general graph $G$, the off-diagonal entries of the resistance matrix are fractions with denominator $\kappa(G)$. 
Thus, we would expect a priori that $\det R$ has $\kappa(G)^n$ as its denominator.
Our theorems show that, to the contrary, the denominator of $\det R$ is at most $\kappa(G)^3$.
This property of the denominator of $\det R$ also extends to weighted graphs, if $\kappa(G)$ is replaced with the spanning tree polynomial $\kappa(G, \alpha)$; see \S\ref{sec:weighted}.
This unexpected cancellation was observed experimentally by Faber \cite[Remark 4.7]{faber},
in connection to arithmetic geometry.

From a computational perspective, the expression \eqref{eq:main-alt} is simpler than \eqref{eq:main} because it avoids the $\kappa_2(G)$ term---in fact, an efficient way to compute the number of two-forests $\kappa_2$ for a general graph, over a brute-force search, is to take advantage of the equality between \eqref{eq:main} and \eqref{eq:main-alt}.

\begin{rmk}
If $G$ is a tree, the effective resistance coincides with shortest-path distance.
Graham and Pollak~\cite{graham-pollak} proved that, in this case,
\begin{equation}\label{eqn:tree-det}
\det R = (-1)^{n-1} 2^{n-2} (n-1). 
\end{equation}
This identity shows that the determinant depends on surprisingly little information of the underlying tree.
Their result led to a large amount of subsequent investigations on tree distance matrices and their generalizations.

We can obtain \eqref{eqn:tree-det} from Theorem~\ref{thm:main} since, for a tree, we have $\kappa(G) = 1$
%$\kappa_2(G) = n - 1$, 
and 
\[
\frac{1}{3}  {\kappa_2(G)} - \frac1{12} \sum_{e \in E(G)} {\kappa (G/ e)^2} = \frac13 (n - 1) - \frac{1}{12} (n - 1) .
\]
It is also clear that \eqref{eqn:tree-det} follows from Theorem~\ref{thm:main-alt}, since for a tree $(R_{ik} - R_{jk})^2 = 1$
given any edge $e = (i, j)$.
\end{rmk}

%This identity was motivated by a problem in data communication,
%and inspired much further research on distance matrices.
%%Self-contained proofs are also given in \cite{du-yeh,yan-yeh}.
%A weighted version of \eqref{eqn:full-det} was proved by Bapat--Kirkland--Neumann~\cite{bapat-kirkland-neumann}.

\begin{rmk}
In other recent work~\cite{RSW-tree-distance}, we generalize the Graham--Pollak tree identity \eqref{eqn:tree-det} to identities for any principal minor of the distance matrix.

There is also an identity for the principal minors of the resistance matrix, in the ongoing work \cite{richman-shokrieh-wu}.
This result is more technical to state, and is motivated by questions from arithmetic geometry~\cite{CR}.
% \cite{baker-rumely}.
\end{rmk}

\begin{rmk}
One particular generalization of \eqref{eqn:tree-det}, found by Bapat--Kirkland--Neumann~\cite{bapat-kirkland-neumann}, is the following identity for the distance matrix of an edge-weighted graph.
If the distance matrix $D$ takes into account positive real weights on edges $\{\alpha_e : e \in E\}$, then
\begin{equation}\label{eqn:w-tree-det}
	\det D = (-1)^{n-1} 2^{n-2} \prod_{e \in E} \alpha_e \sum_{e \in E} \alpha_e .
\end{equation}
There are analogous weighted versions of Theorem~\ref{thm:main} and \ref{thm:main-alt}, see \S\ref{sec:weighted}.
\end{rmk}

%\subsection{Further questions and related work}

%\begin{rmk}
%Discuss -- expression in main thm ``makes sense'' for matroids -- does it mean anything?
%
%For a matroid $M$, on the ground set $E$, let $\kappa(M)$ denote the number of bases.
%Let $\mathrm{tc}_1(M)$ denote the 1-truncation of $M$, whose bases are the independent sets in $M$ of size $\mathrm{rk}(M) - 1$.
%Consider the quantity
%\[
%x(M) = 4 \kappa(M) \kappa(\mathrm{tc}_1(M)) - \sum_{e \in E} \kappa(M / e)^2 .
%\]
%\end{rmk}

\subsection{Organization}

In \S\ref{sec:resistance} we review facts about effective resistance and the resistance curvature of a graph.
In \S\ref{sec:proofs} we prove our main results.
In \S\ref{sec:examples} we provide some examples of resistance matrices of specific graphs.

\section{Effective resistance and graph curvature}\label{sec:resistance}

When working with a graph $G = (V, E)$, we implicitly assume that the vertices are ordered, $V = \{v_1, \ldots, v_n\}$,
and that the edges come equipped with a chosen orientation, and we write $e^+, e^-$ for the endpoints of the edge $e$.
The choice of order and orientation is arbitrary, and will not have any effect on the results.

For notation, we use $r(v_i, v_j)$ to denote the effective resistance between vertices $v_i$ and $v_j$.
The {\em resistance matrix} is the matrix $R$ in $\RR^{|V| \times |V|}$ with entries
$
R_{ij} = r(v_i, v_j).
$
Unless noted otherwise, when discussing effective resistance we always assume that each edge has unit resistance.
As mentioned in the introduction,
\[
	r(v_i, v_j) = \frac{\kappa(G / ij)}{\kappa(G)} ,
\]
where $G / ij$ denotes the graph obtained from gluing together vertices $v_i$ and $v_j$.
For a reference, see Kirchhoff~\cite{kirchhoff} or Biggs~\cite[Section 17]{biggs} for a more modern treatment.
For an edge $e$, the effective resistance between its endpoints is
\begin{equation}\label{eq:res-trees-edge}
	r(e^+, e^-) = \frac{\kappa(G / e)}{\kappa(G)}
\end{equation}
where $G / e$ is the edge contraction; see e.g. \cite[Proposition 17.1]{biggs}.

\subsection{Graph curvature}\label{sec:curvature}

Here we recall the definition of a vector in $\RR^{|V|}$ that has a special relation to the resistance matrix.
The following terminology is due to Devriendt--Lambiotte~\cite[Definition 1, p. 5]{devriendt-lambiotte}.
The same vector, up to a scaling, appeared earlier in Bapat~\cite[p. 76]{bapat-resmatrix}.

\begin{dfn}
The {\em resistance curvature} of a graph $G = (V, E)$ is the vector  $\boldmu \in \RR^{|V|}$ with components
	\begin{equation}
	\boldmu_i = 1 - \frac{1}{2} \sum_{e \in N(v_i)} r(e^+, e^-),
	\end{equation}
where $N(v_i)$ denotes the set of edges incident to vertex $v_i$.
\end{dfn}

We will often abbreviate ``resistance curvature'' to just ``curvature'' of $G$.

\begin{prop}\label{prop:curv-res}
The {curvature} $\boldmu$ of a graph $G$ is the unique vector $\boldmu \in \RR^{|V|}$ that satisfies the two conditions
\begin{enumerate}[(a)]
\item $R \boldmu = \lambda \bone$ for some real number $\lambda$;
\item $\bone^\tr \boldmu = 1$, i.e. the entries of $\boldmu$ sum to one.
\end{enumerate}
\end{prop}
\begin{proof}
Let $\boldtau = 2\boldmu$.
In \cite[Equation (10)]{bapat-resmatrix} Bapat shows that
$\displaystyle
	R \boldtau = \frac{\boldtau^\tr R \boldtau}{2} \bone,
$
which is equivalent to 
$
	R \boldmu = (\boldmu^\tr R \boldmu) \bone .
$
In  \cite[p. 77]{bapat-resmatrix} it is shown that $\bone^\tr \boldtau = 2$.

Alternatively, see Devriendt--Lambiotte~\cite[Equation (22), p. 19]{devriendt-lambiotte}.
\end{proof}

\subsection{Determinant expressions}

\begin{thm}[{Bapat~\cite[Theorem 4]{bapat-resmatrix}}]\label{thm:bapat-det}
    If $R$ is the resistance matrix of a graph, and $\boldmu$ is the curvature, then
    $\displaystyle
        \det R = \frac{(-2)^{n - 1}}{\kappa(G)} \boldmu^\tr R \boldmu .
    $
\end{thm}

In Bapat's notation, this result in stated in terms of $\boldtau = 2\boldmu$, with the identity
\[
	\det R = \frac{(-2)^{n - 3}}{\kappa(G)} \boldtau^\tr R \boldtau .
\]
Bapat proves Theorem~\ref{thm:bapat-det} by first finding an expression for the inverse matrix $R^{-1}$, and then computing the determinant of $R^{-1}$.
Here, for the convenience of the reader, we sketch an alternative argument, along the lines used in \cite{RSW-tree-distance}.
\begin{proof}[Proof of Theorem~\ref{thm:bapat-det}]
First, note that
\[
	\frac{\det R}{\cof R} = \frac{1}{\bone^\tr (R^{-1}) \bone},
\]
where $\cof A$ denotes the {\em sum of cofactors} $\sum_{i, j} (-1)^{i + j} \det A_{i, j}$.
Then, 
we may argue that
\[
	\frac{1}{\bone^\tr (R^{-1}) \bone} = \max\{ \boldu^\tr R \boldu : \boldu \in \RR^V,\, \bone^\tr \boldu = 1 \}
\]
by using the method of Lagrange multipliers and the fact that $R$ has signature $(1, n-1)$.\footnote{We can avoid reference to the signature of $R$ by replacing ``maximum'' with ``critical value''.}
Furthermore, the maximum value of $\boldu \mapsto \boldu^\tr R \boldu $ is achieved precisely at the curvature vector $\boldmu$ due to Proposition~\ref{prop:curv-res}.
It then follows that
\[
	\frac{\det R}{\cof R} = \boldmu^\tr R \boldmu .
\]
Finally, we use the sum-of-cofactors identity $\cof R = (-2)^{n - 1} / \kappa(G)$ which is a special case of \cite[Theorem 7]{bapat-sivasubramanian}.
\end{proof}
%\end{rmk}

\section{Proofs}\label{sec:proofs}
We now prove our main theorem.

\begin{proof}[Proof of Theorem~\ref{thm:main}]
A result of Bapat~\cite[Theorem 4]{bapat-resmatrix} states that
\begin{equation}\label{eq:bapat-det}
	\det R = \frac{(-2)^{n - 1}}{\kappa(G)} \boldmu^\tr R \boldmu ,
\end{equation}
where $\boldmu$ is the curvature vector (see \S\ref{sec:curvature}).
By \cite[Proposition 5.3]{RSW-two-forest}, we have the identity
\begin{equation}\label{eq:gamma-1}
	\boldmu^\tr R \boldmu = 2 \gamma(G) \; := \;  \max\{ \boldu^\tr R \boldu : \boldu \in \RR^V,\, \bone^\tr \boldu = 1 \}.
\end{equation}
Finally, using \cite[Theorem A]{RSW-two-forest}  we have
\begin{equation}\label{eq:gamma-2}
	\gamma(G) = \frac{1}{3} \frac{\kappa_2(G)}{\kappa(G)} - \frac{1}{12} \sum_{e \in E} r(e^+, e^-)^2.
\end{equation}
Together, \eqref{eq:bapat-det}, \eqref{eq:gamma-1}, and \eqref{eq:gamma-2} show that
\[
	\det R = \frac{(-1)^{n - 1} 2^{n}}{\kappa(G)} \left(\frac{1}{3} \frac{\kappa_2(G)}{\kappa(G)} - \frac{1}{12} \sum_{e \in E} r(e^+, e^-)^2 \right) .
\]
This proves the desired result, after making use of the resistance identity \eqref{eq:res-trees-edge} for $r(e^+, e^-)$.
\end{proof}

\begin{proof}[Proof of Theorem~\ref{thm:main-alt}]
We follow the same steps as above, except that in place of \eqref{eq:gamma-2} we instead use \cite[Proposition 5.4 (c)]{RSW-two-forest}
\begin{equation}
	\gamma(G) = \frac{1}{4} \sum_{e \in E} (r(e^+,q) - r(e^-, q))^2 . \qedhere
\end{equation}
\end{proof}

\subsection{Edge-weighted graphs}\label{sec:weighted}

%\begin{rmk}\label{rmk:weighted}
If $G = (V, E)$ is a connected graph equipped with edge resistances $\{ \alpha_e : e \in E \}$, then the effective resistance $R_{ij}$ between vertices $v_i$ and $v_j$ satisfies
$\displaystyle
	R_{ij} = \frac{\kappa(G/ij, \alpha)}{\kappa(G, \alpha)},
$
where 
\[\displaystyle 
	\kappa(G, \alpha) = \sum_{T \in \trees(G)} \prod_{e \not \in T} \alpha_e 
\]
and $\trees(G)$ denotes the set of spanning trees of $G$,
and $G / ij$ denotes the same quotient graph as earlier.
Similarly, we let $\kappa_2(G, \alpha)$ denote the weighted sum 
\[\displaystyle 
	\kappa_2(G, \alpha) = \sum_{F \in \forests_2(G)} \prod_{e \not \in F} \alpha_e .
\]
over the set of two-component spanning forests $\forests_2(G)$ \cite[Sections 15-17]{biggs}.
%\begin{thm}
%\label{thm:w-main}
%Suppose $G = (V,E)$ is a finite graph, with edge resistances $\{\alpha_e : e \in E\}$.
Then, we have the identity
\begin{equation}\label{eq:w-main}
    \det R =  \frac{(-1)^{n-1} 2^{n} {\prod_{E} \alpha_e}}{\kappa(G, \alpha)}  \left( \frac13 \frac{\kappa_2(G, \alpha)}{\kappa(G, \alpha)} - \frac{1}{12} \sum_{e\in E} \alpha_e \frac{\kappa(G/e, \alpha)^2 }{\kappa(G, \alpha)^2} \right).
\end{equation}
Alternatively, the weighted version of Theorem~\ref{thm:main-alt} is
\begin{equation}\label{eq:w-main-alt}
    \det R = \frac{(-1)^{n-1} 2^{n - 2} \prod_{E} \alpha_e}{\kappa(G, \alpha)} \sum_{\substack{e\in E \\ e = (i, j)}} \frac{(R_{ik} - R_{jk})^2 }{\alpha_e} ,
\end{equation}
for any vertex index $k$.
Some examples (Examples~\ref{ex:triangle} and \ref{ex:theta}) are given in \S\ref{sec:examples}.
%\end{rmk}

\subsection{Edge-transitive graphs}

We note that Theorem~\ref{thm:main} can be further simplified if $G$ is sufficiently symmetric.
Recall that $G$ is {\em edge-transitive} if, for any $e, e' \in E(G)$, there is some automorphism of $G$ that sends $e$ to $e'$.
For a nice exposition on edge-transitive graphs, and other symmetry classes, see Biggs~\cite[Chapter 15]{biggs-graph}.

The next result follows from Theorem~\ref{thm:main}.
\begin{cor}\label{cor:edge-transitive}
If $G$ is an edge-transitive graph with $n$ vertices and $m$ edges, then
\begin{equation}
	\det R = \frac{(-1)^{n - 1} 2^n}{\kappa(G)} \left(\frac{1}{3} \frac{\kappa_2(G)}{\kappa(G)} - \frac{1}{12} \frac{(n - 1)^2}{m}\right) . 
\end{equation}
\end{cor}
\begin{proof}
Since $G$ is edge-transitive, $\kappa(G / e) = \kappa(G / e')$ for any two edges $e, e' \in E(G)$.
It follows from a straightforward counting argument that in this case,
$\displaystyle
\frac{\kappa(G / e)}{\kappa(G)} = \frac{|V| - 1}{|E|}
$ 
for every edge $e$.
\end{proof}

% In this case, the denominator of $\det R$ must be a divisor of $\mathrm{lcm}(3 \kappa^2, 12 \kappa |E|) = 3 \kappa \cdot \mathrm{lcm}(\kappa, 4 |E|)$.

\section{Examples}\label{sec:examples}

\begin{eg}[House graph]
Consider the house graph, shown in Figure~\ref{fig:house}.
This graph has $\kappa(G) = 11$ spanning trees, and $\kappa_2(G) = 19$ two-component spanning forests.

\begin{figure}[h]
\centering
	\begin{tikzpicture}[scale=0.5]
	\coordinate (1) at (0,1.8);
	\coordinate (2) at (0,0);
	\coordinate (3) at (2,0);
	\coordinate (4) at (2,1.8);
	\coordinate (5) at (1,3);
	
	\foreach \c in {1,2,3,4,5} {
		\filldraw[black] (\c) circle (2pt);
	}

	\node[left] at (1) {$1$};
	\node[left] at (2) {$2$};
	\node[right] at (3) {$3$};
	\node[right] at (4) {$4$};
	\node[above right] at (5) {$5$};
	
	\draw (1) -- (2) -- (3) -- (4) -- cycle;
	\draw (1) -- (5) -- (4);
	\end{tikzpicture}
	\caption{House graph.}
	\label{fig:house}
\end{figure}
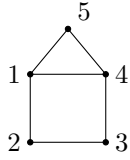

Its resistance matrix is
\[
	R = \frac{1}{11}\begin{pmatrix}
	0 & 8 & 10 & 6 & 7 \\
	8 & 0 & 8 & 10 & 13 \\
	10 & 8 & 0 & 8 & 13 \\
	6 & 10 & 8 & 0 & 7 \\
	7 & 13 & 13 & 7 & 0
	\end{pmatrix} .
\]
A priori, we would expect that $\det R$ has $11^5$ as its denominator.
Instead, due to cancellation in accordance with Theorem~\ref{thm:main} we have
\begin{equation*}
	\det R = \frac{1360}{11^3} \,.
\end{equation*}
In terms of the theorem,
\[
	\det R = \frac{2^5}{11} \left(\frac{1}{3} \cdot \frac{19}{11} - \frac{1}{12} \cdot\frac{6^2 + 7^2 + 7^2 + 8^2 + 8^2 + 8^2}{11^2}\right) .
\]
Applying Theorem~\ref{thm:main-alt} with $k = 1$, we have
\[
	\det R = \frac{2^3}{11} \left(\frac{8^2 + 6^2 + 7^2 + (10 - 8)^2 + (10 - 6)^2 + (7 - 6)^2}{11^2}\right).
\]
\end{eg}

\begin{eg}[Cube graph]
The cube graph has 8 vertices and 12 edges; see Figure~\ref{fig:cube}.
This graph has $\kappa(G) = 384$ spanning trees.
For the number of two-forests, we have:
\begin{align*}
 \kappa_2(G) &= \#(\text{all 6-edge subgraphs}) - \#(\text{those containing cycles}) \\
 &= \binom{12}{6} - \left(6 \cdot \binom{8}{2} + 12 + 4\right) = 740 .
\end{align*}
Its resistance matrix is
\[
	R = \frac{1}{12}\begin{pmatrix}
	0 & 7 & 9 & 7 &  9 & 10 & 9 & 7 \\
	7 & 0 & 7 & 9 &  10 & 9 & 7 & 9 \\
	9 & 7 & 0 & 7 &  9 & 7 & 9 & 10 \\
	7 & 9 & 7 & 0 &  7 & 9 & 10 & 9 \\
	
	9 & 10 & 9 & 7 &  0 & 7 & 9 & 7 \\
	10 & 9 & 7 & 9 &  7 & 0 & 7 & 9 \\
	9 & 7 & 9 & 10 &  9 & 7 & 0 & 7 \\
	7 & 9 & 10 & 9 &  7 & 9 & 7 & 0
	\end{pmatrix} .
\]
We have $\displaystyle \det R = - \frac{86\, 593\, 536}{12^8} = - \frac{29}{12^2}$.

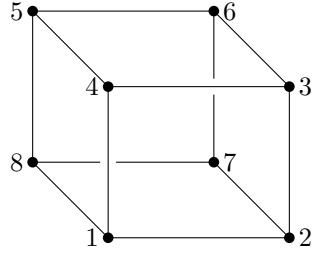
\begin{figure}[h]
\centering
\begin{tikzpicture}
	\coordinate (A) at (0,0);
	\coordinate (C) at (2.4,0);
	\coordinate (D) at (0,2);
	\coordinate (F) at (2.4,2);
	\coordinate (AA) at (-1,1);
	\coordinate (CC) at (1.4,1);
	\coordinate (DD) at (-1,3);
	\coordinate (FF) at (1.4,3);
	
	\draw (AA) -- (CC) -- (FF) -- (DD) -- cycle;
	% \draw[line width=6pt, color=white] (E) -- (G) -- (GG);
	% \draw (BB) -- (EE);
	\draw[line width=6pt, color=white] (A) -- (D) -- (F);
	\draw (A) -- (C) -- (F)  -- (D) -- cycle;
	% \draw (B) -- (E);
	\draw (A) -- (AA);
	\draw (C) -- (CC);
	\draw (D) -- (DD);
	\draw (F) -- (FF);
	% \draw (G) -- (GG);
	
	\foreach \p in {A,C,D,F,AA,CC,DD,FF} {
		\fill (\p) circle (2pt);
	}
	\foreach \p/\n/\dir in {A/1/left,C/2/right,F/3/right,D/4/left,DD/5/left,FF/6/right,CC/7/right,AA/8/left} {
		\node[\dir] at (\p) {\n};
	}
\end{tikzpicture}
\caption{Cube graph}
\label{fig:cube}
\end{figure}

Note that the cube graph is edge-transitive.
In terms of Corollary~\ref{cor:edge-transitive},
\[
	\det R = - \frac{2^8}{384} \left(\frac{1}{3} \cdot \frac{740}{384} - \frac{1}{12} \cdot \frac{7^2}{12}\right) .
\]
In terms of Theorem~\ref{thm:main-alt}, with $k = 1$, we have
\[
	\det R = - \frac{2^6}{384} \left(\frac{3 \times 7^2 + 6 \times 2^2 + 3 \times 1^2}{12^2} \right) .
\]
\end{eg}

\begin{eg}[Triangle graph]\label{ex:triangle}
Suppose $G$ is the triangle graph with general edge weights $a$, $b$, and $c$.
Then the resistance matrix is
\[
    R = \frac{1}{a + b + c} \begin{pmatrix}
        0 & ab + ac & ab + bc \\
        ab + ac & 0 & ac + bc \\
        ab + bc & ac + bc & 0
    \end{pmatrix}.
\]
Its determinant is
\[
    \det R = \frac{2 abc (a + b)(a + c)(b + c)}{(a + b + c)^3} .
\]
The curvature vector for this graph is 
\[
    \boldmu = \frac{1}{2(a + b + c)}\begin{pmatrix}
        a + b \\
        a + c \\
        b + c
    \end{pmatrix}.
\]
This graph has $\kappa(G; \alpha) = a + b + c$ and $\kappa_2(G; \alpha) = ab + ac + bc$.
The weighted version of the main theorem \eqref{eq:w-main} states that we have
\[
	\det R = 2^3 \cdot \frac{abc}{a + b + c} \left(\frac13 \cdot \frac{ab + ac + bc}{a + b + c} - \frac{1}{12} \cdot \frac{a(b + c)^2 + b(a + c)^2 + c(a + b)^2}{(a + b + c)^2} \right)
\]
Alternatively, if we use \eqref{eq:w-main-alt} with $k = 1$,
%the weighted version of Theorem~\ref{thm:main-alt}, 
we have
\[
	\det R = 2 \cdot \frac{abc}{a + b + c} \left( \frac{1}{a}\cdot \frac{(ab + ac)^2}{(a + b + c)^2} + \frac{1}{b}\cdot \frac{(ab + bc)^2}{(a + b + c)^2} + \frac{1}{c}\cdot \frac{(ac - bc)^2}{(a + b + c)^2}\right) .
\]
\end{eg}

\begin{eg}[Theta graph]\label{ex:theta}
Consider the theta graph, i.e. the unique multigraph with two vertices, three edges, and no loops.
For the theta graph with general edge weights $a$, $b$, and $c$, the resistance matrix is
\[
	R = \frac{1}{ab + ac + bc} \begin{pmatrix}
	0 & abc \\
	abc & 0
	\end{pmatrix}.
\]
so $\displaystyle \det R = - \left(\frac{abc}{ab + ac + bc}\right)^2$.
This graph has $\kappa(G; \alpha) = ab + ac + bc$ and $\kappa_2(G; \alpha) = abc$.
The main theorem \eqref{eq:w-main} states that we have
\[
	\det R = -2^2 \cdot \frac{abc}{ab + ac + bc} \left(\frac13 \cdot \frac{abc}{ab + ac + bc} - \frac{1}{12} \cdot \frac{a b^2 c^2 + a^2 b c^2 + a^2 b^2 c}{(ab + ac + bc)^2} \right) .
\]
\end{eg}

%\begin{eg}[Dipole graph]
%Consider the multigraph with two vertices, $n$ edges with $n \geq 3$, and no loops.
%This graph is known as the dipole graph, banana graph, or multi-loop sunset graph \todo{cite} in various places in the literature.
%
%For the dipole graph with edge weights $\alpha_1, \ldots, \alpha_n$, 
%the weighted counts of spanning trees and two-forests are
%$\kappa(G; \alpha) = \sum_{i = 1}^n \prod_{j \neq i} \alpha_j$
%and $\kappa_2(G; \alpha) = \prod_{i = 1}^n \alpha_i$.
%
%the resistance matrix is
%\[
%	R = \frac{1}{\kappa} \begin{pmatrix}
%	0 & \kappa_2 \\
%	\kappa_2 & 0
%	\end{pmatrix}
%\]
%where ...
%so $\det R = - (\kappa_2 / \kappa)^2$.
%The weighted version of the main theorem \eqref{eq:w-main} states that
%\[
%	\det R = \frac{- 2^2 \prod_i \alpha_i}{\kappa(G; \alpha)}\left ( \frac13 \frac{\kappa_2}{\kappa} - \frac{1}{12} \frac{\alpha_1 \prod_{j \neq 1} \alpha_j + \alpha_2 (\prod_{j \neq 2} \alpha_j)^2 + \cdots}{\kappa^2} \right)
%\]
%\end{eg}

\section*{Acknowledgments}
HR would like to thank Yen-Jen Cheng, Sen-Peng Eu, and Yuan-Hsun Lo for their invitation to speak at the ILAS 2025 Minisymposium on ``Enumerative and algebraic combinatorics and matrices'' at National Sun Yat-Sen University.
The authors would like to thank Ravindra Bapat and Karel Devriendt for helpful discussion.
FS was partially supported by NSF CAREER DMS-2044564 grant.
CW was partially supported by Simons Collaboration Grant 850685.

\bibliography{resistance-det-ref} 
\bibliographystyle{abbrv}

\end{document}